\documentclass[11pt]{amsart}
\usepackage{amssymb}
\usepackage{amsthm}
\usepackage{latexsym,array}
\usepackage{amsthm}
\usepackage{amsmath, mathrsfs}
\usepackage[english]{babel}
\usepackage{color}
\usepackage[mathscr]{euscript}
\usepackage{hyperref}
\usepackage{url}
\usepackage{cite}
\usepackage{latexsym,enumerate}
\usepackage{amsmath,amsthm,amsopn,amstext,amscd,amsfonts,amssymb}
\usepackage{verbatim}

\usepackage[left=2.5cm,right=2.5cm,top=2cm,bottom=2cm]{geometry}

\newtheorem{theorem}{Theorem}[section]
\newtheorem{lemma}[theorem]{Lemma}

\newtheorem{corollary}[theorem]{Corollary}
\newtheorem{definition}[theorem]{Definition}

\newtheorem{remark}[theorem]{Remark}

\newcommand{\spb}[1]{\smallskip}
\newcommand{\mpb}[1]{\medskip}
\newcommand{\bpb}[1]{\bigskip}

\begin{document}
\title[Hypercomplex singular Cauchy-like integral in Douglis analysis]{A note on the hypercomplex Singular Cauchy-like Integral in Douglis Analysis}

\author[I. Paulino-Basurto]{Isidro Paulino-Basurto}
\address{SEPI, ESFM-Instituto Polit\'ecnico Nacional. 07338, Ciudad M\'exico, M\'exico}
\email{isidroo-@hotmail.com}
\author[J. Marie Vilaire]{Jean Marie Vilaire}
\address{Department of Computer and Software Engineering and Mobile Computing and Networking Research
Laboratory, Polytechnique, Montréal, Montreal, Quebec, Canada}
\email{jeanmarievilaire@yahoo.fr}
\author[J. O. González-Cervantes]{José Oscar González-Cervantes}
\address{SEPI, ESFM-Instituto Polit\'ecnico Nacional. 07338, Ciudad M\'exico, M\'exico}
\email{jogc200678@gmail.com}
\author[J. Bory-Reyes]{Juan Bory-Reyes}
\address{SEPI, ESIME-Zacatenco-Instituto Polit\'ecnico Nacional. 07338, Ciudad M\'exico, M\'exico}
\email{juanboryreyes@yahoo.com}

\begin{abstract}
\noindent
Hypercomplex analysis extends classical complex analysis to Douglis algebras, replacing holomorphic functions with hyperanalytic functions, i.e., null solutions of the Douglis operator in $\mathbb R^2$ (identifying $\mathbb R^2$ with $\mathbb C$). The hypercomplex singular Cauchy-like integral, is one of the main actors in the hypercomplex analysis and many properties have been studied in recently decades. This paper aims to present local estimates of the hypercomplex singular Cauchy-like integral.
\end{abstract}
\maketitle

\vspace{0.3cm}
\small{

}

\section{Introduction}
Classical complex analysis offers refined solutions to a variety of problems and several applications in physics, engineering and many real-world phenomena. However, currently often require richer structures capable of handling new elliptic systems.

In his $1953-$paper \cite {Douglis1953}, Avron Douglis introduced an algebra and a class of functions, which satisfy classically the principal part of an elliptic system of $2r$ equations with $2r$ unknowns and two independent variables, the so-called Douglis system (that is, an elliptic system of first order in two independent variables). 

When the above-mentioned elliptic system is considered homogeneous and contains no terms of zero order, can be represented by a single hypercomplex equation in the matter of the so-called Douglis operator, whose continuously differentiable null solutions are called hyperanalytic functions. As a generalization of the holomorphic functions in the complex plane, hyperanalytic functions are the main objects in hypercomplex analysis. The classical works here are \cite{Gilbert1983, GilbertBuchanan1989, GilbertHile1974, GilbertWendland1975, Soldatov2006, Wendland1979}. 

Related results to hypercomplex function theory and development of techniques based on Cauchy-like integrals in Douglis algebras context can be found in \cite{BlayaBoryPena2003, AbreuBoryPeña2005, AbreuBoryVilaire2009, AbreuBoryVilaire2010, AbreuBoryVilaire2014, GilbertZeng1992, Huang1988, Liu1996}.

\section{Preliminaries}
In this section, we will touch the fundamentals and cornerstones for the hypercomplex analysis, needed for the core of our work.

The Douglis algebra $\mathbb D$ is generated by the two elements $i$ and $e$, governed by the multiplication rules: $i^2=-1$, $ie=ei$, $e^0=1$, $e^r=0$, where $r$ is a positive integer.

A particular element $a\in \mathbb D$ can be written as a hypercomplex number in the form: $a=a_0+\sum\limits_{k=1}^{r-1}a_ke^k,$ where each $a_k$ is a complex number. The expression $A = \sum\limits_{k=1}^{r-1}a_ke^k$ is called the nilpotent part of $a$ and $a_0$ the complex part.

The conjugation of any hypercomplex number $a$, denoted by $\bar{a}$, is defined as $\bar{a} = \sum\limits_{k=0}^{r-1}\bar{a}_k e_k$.

If the complex part $a_0$ of a hypercomplex number $a$ is not null then the multiplicative inverse $a^{-1}$ of $a$ is given by
\[
a^{-1}  =  a_0^{-1} \sum_{k=0}^{r-1}(-1)^k\left( \dfrac{A}{a_0} \right)^k   .
\]

If $a_0 = 0$, then $a$ does not have a multiplicative inverse.

The algebraic norm of $a$ is defined by the expression
\[
\left| a \right| := \sum_{k=0}^{r-1} |a_k|   .
\]
A hypercomplex function $f$ is a map from the complex plane into the Douglis algebra and has the form $f(z):=\sum\limits_{k=0}^{r-1}f_k(z)e^k$, where $f_k$ are complex functions and $f_0$ is called the complex part of $f$.

We say that $f$ belongs to a class of $\mathbb D-$valued functions if its components $f_k$ belong to the corresponding $\mathbb C-$valued functions class. For example, we say that $f\in C^1(\Omega),$ where $\Omega$ is some domain in the complex plane $\mathbb{C},$ if $f$ is defined on $\Omega$ and each $f_k$ and its derivative are continuous in $\Omega.$

The central notion in hypercomplex analysis is that of hyperanalytic function, introduced by means of the Douglis differential operator, i.e.,  the operator given by
\[\partial_{\overline{z}}^q:=\partial_{\overline{z}}+q(z){\partial_z},\]
which has bounded and continuous derivatives on $\mathbb R^2$ (identifying $\mathbb{R}^2$ with $\mathbb{C}$ in the usual way) and vanishing outside of a sufficiently large circle, $q(z)\in\mathbb D$ is a nilpotent function and
\[
\partial_{\overline{z}}f+q(z){\partial_z}f=0, \qquad z=x+iy,
\]
where $\partial_z:=\frac{1}{2}(\partial_x-i\partial_y),\ \partial_{\overline{z}}:=\frac{1}{2}(\partial_x+i\partial_y)$.

A $\mathbb D$-valued function $f$ is said to be \emph{hyperanalytic} in the domain $\Omega$ of $\mathbb C$ if and only if $f\in C^1(\Omega)$ and satisfies in $\Omega$ the equation $\partial_{\overline{z}}^q{f}=0.$

Fundamental solution of the Douglis operator, what is known as the Douglis Cauchy kernel, is given by
\begin{center}
$\displaystyle e_z(\zeta)=\frac{1}{2\pi}\frac{\partial_{\zeta}w(\zeta)}{w(\zeta)-w(z)}\rm,$
\end{center}
being $w(z)=z+\sum\limits_{k=1}^{r-1}w_k(z)e^k$ the generating solution of $\displaystyle \partial_{\overline{z}}^q,$ i.e., $\partial_{\overline{z}}^q{w}=0.$ The nilpotent part of $w$ possesses bounded and continuous derivatives up to order two in the complex plane.

The following properties of the generating solution $w(z)$ will be useful.
\[
|w(z_2)-w(z_1)|\le c|z_2-z_1|,\ |w(z_2)-w(z_1)|^{-1}\le c |z_2-z_1|^{-1}.
\]

Throughout the paper $\gamma$ denotes a closed Jordan curve oriented positively which is the boundary of a simply connected domain $\Omega\ni 0$ in the complex plane $\mathbb{C}.$ We will use the notation $\Omega^+:=\Omega$ and $\Omega^-:=\mathbb C\setminus ({\Omega\cup \gamma})$.

Let us introduce the following notations \( \gamma_{\delta}(t) = \{\tau \in \gamma: |\tau - t| \leq \delta\} \), \( 0 < \delta \leq d \), $\theta(\delta) = \sup_{t \in \gamma} \theta_t(\delta)$. \( \theta_t(\delta) = \text{mes} \, \gamma_{\delta}(t) :=\displaystyle\int_{\gamma_{\delta}(t)}ds\),\; \(\theta(\delta) = \sup_{t \in \gamma} \theta_t(\delta)\). It is clear that \( \theta(\delta) \geq \delta \).

We introduce the class of regular curves, which will be treated in this paper, see \cite{Salaev1976}.
\begin{definition}
Let $\gamma$ be a closed Jordan curve oriented positively. The curve $\gamma$ is said to be regular if there exists a constant $c >1$ such that
\[\theta (\delta )  \leq c \, \delta , \qquad \text{for } \delta\in (0,d].\]
\end{definition}

Examples of regular curves are convex curves, Liapunov curves, chord arc curves, smooth curves and Lipschitz curves.

Now we shall introduce the class of functions which will be used in this paper.
\begin{definition}
    Let $\gamma$ be a regular curve and $0< \alpha \leq 1$. We define the classic Hölder space:
    \[
    H_{\alpha}(\gamma) = \{ f \in C(\gamma) : |f(t) - f(\tau)| \leq c_f |t - \tau|^{\alpha}, \, \forall \ t, \tau \in \gamma\},
    \]
where $C(\gamma)$ denotes the class of all $\mathbb D-$valued continuous functions on $\gamma$.
\end{definition}
We shall consider the modulus of continuity $\omega_f$ of a continuous function $f$ as follows:
\[
w_f (\xi)  =  \xi \sup_{\delta \geq \xi} \delta^{-1} \sup_{t,\tau \in\gamma \ |t-\tau|\leq \delta} |f(t)-f(\tau)| , \qquad \text{for } \xi > 0  .
\]
Then, Hölder space $H_{\alpha}(\gamma)$ equiped with the norm
\[
\| f \|_{\alpha}    :=   \max_{t\in\gamma}|f(t)| + \sup_{\xi\in (0,d]} \dfrac{\omega_f (\xi)}{\xi^{\alpha}}
\]
becomes a Banach space.

\begin{definition}
Let $\gamma$ be a regular curve. The hypercomplex Cauchy type integral of a $\mathbb{D}$-valued function $f\in H_{\alpha}(\gamma)$ is defined by the expression:
\begin{equation}\label{eq:hyper-Cauchy}
\mathcal{C}_{\gamma}f(z)  =   \int_{\gamma}e_z(\zeta)n_q(\zeta)f(\zeta)ds, \qquad z\in \Omega^\pm
\end{equation}
where $n_q(\zeta):=n(\zeta)+q(\zeta)\overline{n}(\zeta)$, $ds$ is the arc length differential in $\gamma$ and $n(\zeta)$ the unit normal vector to $\gamma$ at the point $\zeta$.
\end{definition}

\begin{remark}
It is clear that $\mathcal{C}_{\gamma}f$ is hyperanalytic in $\mathbb{C}$ $\setminus$ $\gamma$, further $\partial^q_{\bar{z}}e_z(\zeta)=0$ in $\mathbb{C}$ $\setminus$ $\gamma$.
\end{remark}

\begin{definition}
Let $\gamma$ be a regular curve and $f\in H_{\alpha}(\gamma)$. The singular hypercomplex Cauchy type integral of $f$ is defined as
\[
\mathcal{S}_{\gamma}f(t) = 2\displaystyle\int_{\gamma}e_t(\zeta)n_q(\zeta)(f(\zeta)-f(t))ds+f(t),\qquad t\in \gamma,
\]
\end{definition}
Throughout this paper notation $c$ will be used for constants which may vary from one occurrence to the next.
		
\section{Auxiliary Lemmas}
We state some necessary technical tools, which are indispensable in estimating inequalities that will be required later. The reader may check for instance \cite{BlayaBoryPena2003} for reference.

Let $t\in\gamma$ and $0< r \leq d = diam\gamma$. Write $\gamma_r (t) := \{ \tau\in\gamma \ : \ |t-\tau |\leq r \}.$
\begin{lemma}\label{lem_inte_tsup}
Let $\gamma$ be a regular curve. If $t_1,\ t_2\in\gamma$ with $|t_2-t_2|=2\varepsilon,\ \varepsilon\in (0,d/2]$
\[\left| \int_{\gamma \setminus \big(\gamma_\epsilon(t_1)\cup \gamma_\epsilon(t_2)\big)} e_{t_1}(\zeta) n_q(\zeta) ds \right| \leq c.\]
\end{lemma}

\begin{corollary}\label{lem_inte_t} 
Let $\gamma$ be a regular curve. If $t \in \gamma$, then
\[
\left| \int_{\gamma \setminus \gamma_\epsilon(t)} e_t(\zeta) n_q(\zeta) ds \right| \leq c.
\] 
\end{corollary}

\begin{lemma}\label{lem_g}
Let $\gamma$ be a regular curve. Let $g$ be a non negative function, non increasing in $(0,d].$ Then for every $r'$, $r''\in (0,d]$ and $r''>r'$, we have the following formula
\[ 
\int_{\gamma_{r''}(x)\setminus \gamma_{r'}(x)}g(|y-x|)ds=\int_{r'}^{r''}g(\xi)d\theta_x(\xi)  .
\]
\end{lemma}

Before passing another section it is important to mention the following result called The Sokhotski-Plemelj formulas which will be useful later.
\begin{theorem} \label{Sokhotski-Plemelj formula}
Let $\gamma$ be a regular curve and let $f\in H_{\alpha}(\gamma )$ then $\mathcal{C}_{\gamma}f$ has continuous limit values in $\gamma$ and moreover
\begin{align}
& \mathcal{C}^+ _{\gamma} f (t)  :=   \lim_{\Omega^+ \ni z \to t}  \mathcal{C}_{\gamma}f(z)  =  \dfrac{1}{2} (\mathcal{S}_{\gamma}f(t) + f(t)), \qquad t\in\gamma \label{SPf1}
\\
& \mathcal{C}^- _{\gamma} f (t)  :=   \lim_{\Omega^- \ni z \to t}  \mathcal{C}_{\gamma}f(z)  =  \dfrac{1}{2} (\mathcal{S}_{\gamma}f(t) - f(t)), \qquad t\in\gamma. \label{SPf2}
\end{align}
\end{theorem}

\begin{remark}
Adding and subtracting equalities \ref{SPf1} and \ref{SPf2} from each other, we get
\begin{align*}
\mathcal{C}^+ _{\gamma} f + \mathcal{C}^- _{\gamma} f   &=   \mathcal{S} _{\gamma} f ,
\\
\mathcal{C}^+ _{\gamma} f - \mathcal{C}^- _{\gamma} f   &=   f  . 
\end{align*}
\end{remark}

\begin{remark}
Another consequence of the Theorem \ref{Sokhotski-Plemelj formula} is the following reasoning.
\\
Since $\gamma$ is a regular curve, from \cite[Lemma 1]{BlayaBoryPena2003} we have
\[
\int_{\gamma} e_t (\zeta) n_q (\zeta)(\mathcal{C}_{\gamma}^+ f(\zeta)  -  \mathcal{C}_{\gamma}^+ f(t)) ds      =     0 ,
\]
which implies
\begin{align*}
&\int_{\gamma} e_t (\zeta) n_q (\zeta)(\mathcal{S}_{\gamma} f(\zeta)  -  \mathcal{S}_{\gamma} f(t)) ds          +         \int_{\gamma} e_t (\zeta) n_q (\zeta)(f(\zeta)  -  f(t)) ds =  0
\\
& \Longrightarrow \quad
\mathcal{S}_{\gamma}^2 f(t) - \mathcal{S}_{\gamma} f(t) + \mathcal{S}_{\gamma} f (t)  -  f(t)    =     0
\quad  \Longrightarrow \quad
\mathcal{S}_{\gamma}^2 f(t)    =    f(t)  .
\end{align*}
This finally yields to the involution property of the singular hypercomplex Cauchy type integral
\begin{equation}\label{involution property}
\mathcal{S}_{\gamma}^2   =   I   ,  
\end{equation}
where $I$ is the identity operator.
\end{remark}
   
\section{Freud theorem}
We provide the following theorem, whose proof found in \cite{BlayaBoryPena2003}, will be useful in the proof of the main theorem of this section. 
\begin{theorem} \label{bounded operator}
Let $\gamma$ be a regular curve. Then the singular hypercomplex Cauchy type integral operator $\mathcal{S}_{\gamma}$ is a bounded operator on $H_{\alpha}(\gamma )$, that is
\[
\| \mathcal{S}_{\gamma} f \|_{\alpha}    \leq   \| f \|_{\alpha} ,
\]
for any $f\in H_{\alpha} (\gamma)$.
\end{theorem}

The following Theorem is one of the paper's central goals: a Douglis-valued analogue of a Freud theorem, proved in \cite{MamedkhanovDadashova2018} in the context of complex analysis, on regular curves. 

\begin{theorem}\label{th_main_freud}
Let \( \gamma \) be regular curve. Let $0<\alpha <1$ and $f\in H_{\alpha}(\gamma )$. Then
\[
m(f, \alpha) = m(\mathcal{S}_\gamma f, \alpha) \quad \text{ almost everywhere},
\]
where  
\[
m(f, \alpha):= \{ t \in \gamma : |f(t) - f(\tau)| = o(|t - \tau|^{\alpha}), \, \tau \to t \}.
\]
\end{theorem}
	
\begin{proof} 
Firstly we will prove
\[ 
m(f, \alpha )  \subset  m(\mathcal{S}_{\gamma} f, \alpha ).
\]
Let \( t \in m(f, \alpha) \). Then, for any \( \varepsilon > 0 \), there exists \( \delta = \delta(\varepsilon, t) \) such that
\[
|t - \tau | < \delta \Rightarrow |f(t) - f(\tau )| \leq \varepsilon |t - \tau |^{\alpha}.
\]

This yields $\omega_f(|t-\tau |) \leq \varepsilon |t - \tau |^{\alpha}$.

The following estimations were described in the proof of Theorem \ref{bounded operator}.
\begin{align*}
|\mathcal{S}_{\gamma} f(t)  -  \mathcal{S}_{\gamma} f(\tau)|
&\leq
c\left(     \int_0^{|t-\tau|} \dfrac{\omega_f (\xi)}{\xi}d\xi  +   |t-\tau| \int_{|t-\tau|}^d \dfrac{\omega_f (\xi)}{\xi^2}d\xi   +   \omega_f (|t-\tau|)    \right)
\\
& \leq c \omega_f (|t-\tau|)
\end{align*}
So we get
\[
\lim_{\tau\to t} \dfrac{|\mathcal{S}_\gamma f(t) - \mathcal{S}_\gamma f(\tau)|}{|t-\tau|^{\alpha}} \leq c\varepsilon ,
\]
and if $\varepsilon\to 0$ the limit gives the following result
\[
\lim_{\tau\to t} \dfrac{|\mathcal{S}_\gamma f(t) - \mathcal{S}_\gamma f(\tau )|}{|t-\tau|^{\alpha}} = 0 ,
\]
which means $|\mathcal{S}_\gamma f(t) - \mathcal{S}_\gamma f(\tau)| = o(|t-\tau|^{\alpha})$. Therefore we get the first contention.
\\
For the another contention we mention that in the Hölder space $H_{\alpha }(\gamma )$ is valid the equality \ref{involution property} $(\mathcal{S}_{\gamma}^2 = I)$ which allows to conclude that
\[
m(\mathcal{S}_{\gamma}f , \alpha)  \subset  m(\mathcal{S}_{\gamma}^2 f , \alpha)  =    m(f, \alpha)  .
\]
	
Therefore \( m(f, \alpha) = m(\mathcal{S}_\gamma f, \alpha) \) almost everywhere.
\end{proof}	

\section{Local Estimates}
In this section, we study the local estimates of the hypercomplex singular Cauchy integral, based on \cite{MamedkhanovJafarov2023}, emphasizing their essential role in understanding the behavior of hyperanalytic functions near particular points. The next Lemma clarifies the conditions under which these local properties appear. 
\begin{lemma}\label{lem aux}
Let $\gamma$ be a regular Jordan curve and $t_0\in\gamma$. Let $ 0<\alpha\le 1,\ \beta\ge 0$ and $\varepsilon=\dfrac{1}{2\,d^{(\alpha+\beta)/2}}$. Then
$f\in H_{\alpha}^{\,\alpha+\beta}(t_0)$ if and only if the following conditions hold:
\begin{enumerate}
\item \label{it:1}
\[
|f(\zeta)-f(t)| \le c\,|\zeta-t|^{\alpha},
\qquad
\text{for }\zeta\in \gamma_{\varepsilon |t-t_0|^{1+\beta/\alpha}}(t),
\]
\item \label{it:2}
\[
|f(\zeta)-f(t)| \le c\,|t-t_0|^{\alpha+\beta},
\qquad
\text{for } \zeta\in\gamma_{\frac{3}{2}|t-t_0|}(t_0)\setminus
\gamma_{\varepsilon |t-t_0|^{1+\beta/\alpha}}(t),
\]
\item \label{it:3}
\[
|f(\zeta)-f(t)| \le c\,|\zeta-t|^{\alpha+\beta},
\qquad
\text{for }\zeta\in \gamma\setminus\gamma_{\frac{3}{2}|t-t_0|}(t_0).
\]
\end{enumerate}
\end{lemma}

\begin{theorem}\label{th:main}
Let $\gamma$ be a regular Jordan curve, fix $t_0 \in \gamma$, $f \in H_{\alpha}^{\alpha+\beta}(t_0)$, $0<\alpha\le 1$, and $\beta \ge 0$.
Then, for the singular integral $S_{\gamma}f(t)$, the following inequality is true:
\[
\bigl| \mathcal{S}_{\gamma}f(t)  -  \mathcal{S}_{\gamma}f(t_0)\bigr|
\le c  \cdot 
\begin{cases}
|t-t_0|^{\alpha+\beta}\ln\!\frac{2d}{|t-t_0|},\ \text{ for}\ \alpha+\beta \le 1\\ \\
|t-t_0|,\ \text{ for}\ \alpha+\beta > 1
\end{cases}
\]
\end{theorem}

\begin{proof}
Let $t\in\gamma$, $t\neq t_0$, and $\varepsilon=\tfrac12 d^{-\frac{\alpha+\beta}{2}}$.  
Then the following identity holds:
\[
\begin{aligned}
	\mathcal{S}_\gamma f(t) - \mathcal{S}_\gamma f(t_0 )&= \dfrac{w(t) - w(t_0)}{\pi} \int \limits_{\gamma \setminus \gamma_{\frac{3}{2} |t-t_0|}(t_0)} \dfrac{\partial_\zeta w(\zeta ) n_q (\zeta ) (f(\zeta ) - f(t_0))}{(w(\zeta )-w(t))(w(\zeta )-w(t_0))} ds 
    \\
	&+ 2\Bigg[ (f(t_0)-f(t))\int \limits_{\gamma\setminus\gamma_{\frac{3}{2}|t-t_0|}(t_0)} e_{t}(\zeta)n_q(\zeta)\, ds
    \\
    &+\left(\int \limits_{\gamma_{\frac{3}{2}|t-t_0|}(t_0)\setminus \gamma_{\varepsilon |t-t_0|^{1+\beta /\alpha}}(t)} + \int \limits_{\gamma_{\varepsilon |t-t_0|^{1+\beta /\alpha}}(t)}
	\right) e_{t}(\zeta)n_q(\zeta)\Big(f(\zeta) - f(t)\Big) \, ds
    \\
	&- \int \limits_{\gamma_{\frac{3}{2}|t-t_0|}(t_0)}e_{t_0}(\zeta)n_q(\zeta) \Big(f(\zeta) - f(t_0)\Big) \, ds\Bigg] + (f(t) - f(t_0)) \stackrel{\text{def}}{=} \sum_{i=1}^6I_i.
\end{aligned}
\]
Let us first estimate $I_1.$ We note for $\zeta\in \gamma\setminus\gamma_{\frac{3}{2}|t-t_0|}(t_0),$ then $|\zeta-t_0|>\frac32 |t-t_0|$ and 
\[
|\zeta-t|\ge |\zeta-t_0|-|t-t_0|\ge |\zeta-t_0|-\frac23 |\zeta-t_0|=\frac13 |\zeta-t_0|.
\]
Since $f\in H_{\alpha}^{\alpha+\beta}(t_0),$ then $|f(\zeta )-f(t_0)|\le c|\zeta - t_0|^{\alpha+\beta},$ therefore,
\begin{equation}\label{eq:I_1estimation}
\begin{aligned}
|I_1|& \leq \dfrac{|w(t) - w(t_0)|}{\pi} \int \limits_{\gamma \setminus \gamma_{\frac{3}{2} |t-t_0|}(t_0)} \dfrac{|\partial_\zeta w(\zeta )|\, |n_q (\zeta )|\, |f(\zeta ) - f(t_0)|}{|w(\zeta )-w(t)|\, |w(\zeta )-w(t_0)|} |ds| 
\\
& \leq  \dfrac{c}{\pi} |t - t_0| \int \limits_{\gamma \setminus \gamma_{\frac{3}{2} |t-t_0|}(t_0)} \dfrac{|f(\zeta ) - f(t_0)|}{|\zeta - t|\, |\zeta - t_0|} |ds|
\\
& \leq \dfrac{3c}{\pi} |t - t_0| \int \limits_{\gamma \setminus \gamma_{\frac{3}{2} |t-t_0|}(t_0)} \dfrac{|f(\zeta ) - f(t_0)|}{|\zeta - t_0|^2} |ds|
\\
& \leq c |t - t_0| \int \limits_{\gamma \setminus \gamma_{\frac{3}{2} |t-t_0|}(t_0)} |\zeta - t_0|^{\alpha + \beta -2} |ds| .
\end{aligned}
\end{equation}
By Lemma \ref{lem_g} we get
\[
|I_1 | \leq c |t-t_0 | \int \limits_{\frac{3}{2} |t-t_0|}^d \xi^{\alpha + \beta -2} d\xi  .
\]
Then we have two cases:
\begin{enumerate}
    \item If $\alpha + \beta < 2$, then 
			 \begin{align*}
		|I_1|	    
        & \leq c   |t-t_0 | \int \limits_{\frac{3}{2} |t-t_0 |}^d  (\xi+|t-t_0 |)^{\alpha + \beta -2}  d\xi
        \leq c   |t-t_0 | \int \limits_0^d  (\xi+|t-t_0 |)^{\alpha + \beta -2}  d\xi
        \\
        & \leq c
        \begin{cases}
                |t-t_0 |^{\alpha + \beta} ,  & \text{if} \quad \alpha + \beta <1 \\
                |t-t_0 | \ln{\dfrac{2d}{|t-t_0 |}} , & \text{if} \quad \alpha + \beta = 1  \\
                |t-t_0 | , & \text{if} \quad 1 < \alpha + \beta < 2  .
        \end{cases}
        \end{align*}
    \item If $\alpha + \beta \geq 2$, then   
    \begin{align*}
    |I_1 | &\leq c|t-t_0 | \int \limits_{0}^d  \xi^{\alpha + \beta -2}  d\xi    \leq c|t-t_0 | \int \limits_0^d  d^{\alpha + \beta -2} d\xi
    \leq c |t-t_0 |d^{\alpha+ \beta -1}    \leq  c |t-t_0 |	.
    \end{align*}
\end{enumerate}
Thus $I_1 $ is bounded as follows,
\begin{align*}
        |I_1 |  
         \leq c
        \begin{cases}
                |t-t_0 |^{\alpha + \beta} ,  & \text{if} \quad \alpha + \beta <1 \\ 
                |t-t_0 | \ln{\dfrac{2d}{|t-t_0 |}} , & \text{if} \quad \alpha + \beta = 1  \\
                |t-t_0 | , & \text{if} \quad  \alpha + \beta > 1  .
        \end{cases}
        \end{align*}
We estimate $I_2$ by noting that
\[
\int \limits_{\gamma\setminus \gamma_{\frac{3}{2}|t-t_0 |}(t_0 )} e_t (\zeta ) n_q (\zeta ) ds
= 
\int \limits_{\gamma\setminus \gamma_{\frac{|t-t_0 |}{2}}(t)} e_t (\zeta ) n_q (\zeta ) ds
-
\int \limits_{\gamma_{\frac{3}{2}|t-t_0 |}(t_0 )\setminus\gamma_{\frac{|t-t_0|}{2}}(t)} e_t (\zeta ) n_q (\zeta ) ds	.
\]
Applying Lemma \ref{lem_inte_t} to the first integral on the right-hand side of the above equality, we obtain
\[
\left|   \int \limits_{\gamma\setminus \gamma_{\frac{|t-t_0 |}{2}}(t)} e_t (\zeta ) n_q (\zeta ) ds     \right|    \leq c    .
\]
Let us estimate the second integral on the right-hand side of the above equality by noting first that 
\[
|\zeta - t|  \geq  \dfrac{|t - t_0|}{2}   ,
\]
then
\begin{align*}
\left| \, \int \limits_{\gamma_{\frac{3}{2}|t-t_0 |}(t_0 )\setminus\gamma_{\frac{|t-t_0|}{2}}(t)} e_t (\zeta) n_q(\zeta)ds  \right|
& = \left| \, \int \limits_{\gamma_{\frac{3}{2}|t-t_0 |}(t_0 )\setminus\gamma_{\frac{|t-t_0|}{2}}(t)} \dfrac{1}{2\pi} \dfrac{\partial_{\zeta}w(\zeta)}{w(\zeta) - w(t)} n_q(\zeta)ds  \right|
\\
& \leq \int \limits_{\gamma_{\frac{3}{2}|t-t_0 |}(t_0 )\setminus\gamma_{\frac{|t-t_0|}{2}}(t)} \dfrac{1}{2\pi} \left|\dfrac{\partial_{\zeta}w(\zeta)}{w(\zeta) - w(t)} \right| |n_q(\zeta)| \, |ds|
\\
& \leq
 c \int \limits_{\gamma_{\frac{3}{2}|t-t_0 |}(t_0 )\setminus\gamma_{\frac{|t-t_0|}{2}}(t)} \dfrac{|ds|}{|w(\zeta) - w(t)|}
\\
& \leq
 c \int \limits_{\gamma_{\frac{3}{2}|t-t_0 |}(t_0 )\setminus\gamma_{\frac{|t-t_0|}{2}}(t)} \dfrac{|ds|}{|\zeta - t|}
\\
& \leq
 c \int \limits_{\gamma_{\frac{3}{2}|t-t_0 |}(t_0 )\setminus\gamma_{\frac{|t-t_0|}{2}}(t)} \dfrac{2|ds|}{|t - t_0|}
 \\
& \leq
\dfrac{c}{|t-t_0|} \int \limits_{\gamma_{\frac{3}{2}|t-t_0 |}(t_0 )\setminus\gamma_{\frac{|t-t_0|}{2}}(t)} |ds|
\\ 
&\leq 
\dfrac{c}{|t-t_0|} mes (\gamma_{\frac{3}{2}|t-t_0|}(t_0)) \leq c	.
\end{align*}
So, in conclusion we get
\[
\left| \, \int \limits_{\gamma\setminus \gamma_{\frac{3}{2}|t-t_0 |}(t_0 )} e_t (\zeta) n_q(\zeta)ds  \right|    \leq  c.
\]
Therefore, for any $\alpha$, $\beta$ both the fact $f\in H_{\alpha}^{\alpha + \beta} (t_0 )$ and the last estimations imply that
\[
|I_2| \leq  c |f(t_0) - f(t)|  \leq c|t-t_0|^{\alpha + \beta}.
\] 
Let us estimate $I_3$ by applying property 2 of Lemma \ref{lem aux} 
\begin{align*}
|I_3| 
&= 
\left|  \int \limits_{\gamma_{\frac{3}{2}|t-t_0 |}(t_0 ) \setminus \gamma_{\varepsilon |t-t_0 |^{1+\beta /\alpha}}(t)} e_t (\zeta) n_q(\zeta)(f(\zeta)-f(t))ds   \right|
\\
& \leq
\int \limits_{\gamma_{\frac{3}{2}|t-t_0 |}(t_0 ) \setminus \gamma_{\varepsilon |t-t_0 |^{1+\beta /\alpha}}(t)} \dfrac{1}{2\pi} \left| \dfrac{\partial_{\zeta}w(\zeta)}{w(\zeta) - w(t)} \right| |n_q(\zeta)| \, |f(\zeta)-f(t)| \, |ds|  
\\
&\leq 
c    \int \limits_{\gamma_{\frac{3}{2}|t-t_0 |}(t_0 ) \setminus \gamma_{\varepsilon |t-t_0 |^{1+\beta /\alpha}}(t)} \dfrac{|f(\zeta)-f(t)|}{|w(\zeta) - w(t)|}|ds|
\\
& \leq
c  |t-t_0|^{\alpha + \beta}   \int \limits_{\gamma_{\frac{3}{2}|t-t_0 |}(t_0 ) \setminus \gamma_{\varepsilon |t-t_0 |^{1+\beta /\alpha}}(t)} \dfrac{|ds|}{|\zeta-t|} .
\end{align*}
Then, it follows from Lemma \ref{lem_g} that
\begin{align*}
|I_3|
& \leq  c  |t-t_0|^{\alpha + \beta}   \int \limits_{\varepsilon |t-t_0 |^{1+\beta /\alpha}}^{\frac{3}{2}|t-t_0 |} \dfrac{d\xi}{\xi} 
=
c  |t-t_0|^{\alpha + \beta} \ln \xi \Big|_{\varepsilon |t-t_0 |^{1+\beta /\alpha}}^{\frac{3}{2}|t-t_0 |}
\\
& \leq c  |t-t_0|^{\alpha + \beta} [\ln{(2d)} - \ln{(|t-t_0 |)}]\leq c  |t-t_0|^{\alpha + \beta} \ln{\frac{2d}{|t-t_0 |}},
\end{align*}
and it is valid for any $0 <\alpha \leq 1$, $\beta\geq 0$.
\\
Let us estimate integral $I_4$. Since $\zeta\in \gamma_{\varepsilon |t-t_0|^{1+\beta /\alpha}}(t)$ according to property $1$ of Lemma \ref{lem aux} we have
\[
|f(\zeta) - f(t)| \leq c |\zeta-t|^{\alpha},
\]
then
\begin{align*}
|I_4|   
&= 
\left|  \int \limits_{\gamma_{\varepsilon |t-t_0 |^{1+\beta /\alpha}}(t)} e_t (\zeta) n_q(\zeta)(f(\zeta)-f(t))ds   \right|
\\
& \leq
\int \limits_{\gamma_{\varepsilon |t-t_0 |^{1+\beta /\alpha}}(t)} \dfrac{1}{2\pi} \left| \dfrac{\partial_{\zeta}w(\zeta)}{w(\zeta) - w(t)} \right| |n_q(\zeta)| \, |f(\zeta)-f(t)| \, |ds|  
\\
&\leq 
c    \int \limits_{\gamma_{\varepsilon |t-t_0 |^{1+\beta /\alpha}}(t)} \dfrac{|f(\zeta)-f(t)|}{|w(\zeta) - w(t)|}|ds|
\leq
c    \int \limits_{\gamma_{\varepsilon |t-t_0 |^{1+\beta /\alpha}}(t)}  |\zeta-t|^{\alpha - 1} |ds|   .
\end{align*}
By Lemma \ref{lem_g} we obtain
\[
|I_4 |    \leq   c\int \limits_{0}^{\varepsilon |t-t_0 |^{1+\beta /\alpha}} \xi^{\alpha -1} d\xi.
\]
\\
Thus for any $0 < \alpha \leq 1$, $\beta \geq 0$ we get
\begin{align*}
|I_4 |    
\leq  
\dfrac{c}{\alpha}  \left(   \varepsilon |t-t_0 |^{1+\beta /\alpha}  \right)^{\alpha}
\leq 
c    |t-t_0 |^{\alpha + \beta} .
\end{align*}
Now we estimate $I_5$. Since $f\in H_{\alpha}^{\alpha + \beta} (t_0)$ then 
\[
|f(\zeta)- f(t_0)|  \leq c  |\zeta - t_0|^{\alpha + \beta}   ,
\]
holds for any $0 <\alpha \leq 1$, $\beta\geq 0$. So
\begin{align*}
    |I_5 |   
    &=   
    \left| \,  \int \limits_{\gamma_{\frac{3}{2}|t-t_0 |}(t_0 )} e_{t_0} (\zeta) n_q(\zeta)(f(\zeta)-f(t_0 ))ds  \right|
    \\
    & \leq   
    \int \limits_{\gamma_{\frac{3}{2}|t-t_0 |}(t_0 )} \dfrac{1}{2\pi} \left|  \dfrac{\partial_{\zeta}w(\zeta)}{w(\zeta) - w(t_0)} \right| |n_q(\zeta)| \, |f(\zeta)-f(t_0 )| \, |ds| 
    \\
    & \leq 
    c \int \limits_{\gamma_{\frac{3}{2}|t-t_0 |}(t_0 )} \dfrac{|f(\zeta) - f(t_0)|}{|w(\zeta) - w(t_0)|} |ds|
    \leq c 
    \int \limits_{\gamma_{\frac{3}{2}|t-t_0 |}(z_0 )} |\zeta - t_0 |^{\alpha + \beta - 1} |ds|  
    \\
    &\leq
    c \int \limits_{0}^{\frac{3}{2}|t-t_0 |} \xi^{\alpha + \beta - 1} d\xi,
\end{align*}
where the last inequality is a consequence of Lemma \ref{lem_g}. Therefore

\begin{align*}
    |I_5 |
    & \leq   
    c \int \limits_{0}^{\frac{3}{2}|t-t_0 |} \xi^{\alpha + \beta - 1} d\xi  
     = c \left. \dfrac{\xi^{\alpha + \beta}}{\alpha + \beta} \right|_{0}^{\frac{3}{2}|t-t_0 |}
     = \dfrac{c \left( \dfrac{3}{2} \right)^{\alpha + \beta}}{\alpha + \beta} |t-t_0|^{\alpha + \beta} \leq c |t-t_0 |^{\alpha + \beta}.
\end{align*}

Finally for $I_6$ we have
\[
|I_6 |   =    |f(t) - f(t_0)|  \leq c |t-t_0 |^{\alpha + \beta}   ,
\]
is valid for any $0 <\alpha \leq 1$, $\beta\geq 0$.
        
Once we have bounded each $I_k$, for $k = 1, \ldots , 6$ we get the following estimation        
\begin{align*}
        |\mathcal{S}_{\gamma} f(t)  -  \mathcal{S}_{\gamma} f (t_0 ) |    
         \leq 
         & c  \left(  |t-t_0|^{\alpha + \beta}  +  |t-t_0|^{\alpha + \beta} \ln \dfrac{2d}{|t-t_0 |} \right)
        \\
        & + c
        \begin{cases}
                |t-t_0 |^{\alpha + \beta} ,  & \text{if} \quad \alpha + \beta <1 \\
                |t-t_0 | \ln{\dfrac{2d}{|t-t_0 |}} , & \text{if} \quad \alpha + \beta = 1  \\
                |t-t_0 | , & \text{if} \quad  \alpha + \beta > 1  .
        \end{cases}
        \end{align*}       
In conclusion, taking into account that $1< \ln \dfrac{2d}{|z-z_0 |}$ we get the desired result.
\end{proof}
\section{Future Work}
Some inverse estimates for principal values of complex Cauchy-type integrals has been proved recently in \cite{MamedkhanovJafarov2026}. Extending these inverse estimates to the hypercomplex analysis context is the subject of ongoing work.

\section*{Statements and Declarations}
\subsection*{Funding} \leavevmode \\
The work was partially supported by Instituto Polit\'ecnico Nacional (grant numbers IND-2026–0101, SIP20260491) and SECIHTI (grant number 244139).
\subsection*{Conflicts of interest} \leavevmode \\
No competing interests appear to influence the work reported in this paper are disclosed by the authors.
\subsection*{Author contributions} \leavevmode \\
All authors provided critical feedback and contributed to the drafting of the paper and approved the final manuscript. 
\subsection*{ORCID}\leavevmode \\
\noindent
Isidro Paulino-Basurto: \url{https://orcid.org/0009-0003-4550-4851}
\\
Jean Marie Vilaire: \url{https://orcid.org/0009-0002-7193-468X}
\\
Jos\'e Oscar Gonz\'alez-Cervantes: \url{https://orcid.org/0000-0003-4835-5436}
\\
Juan Bory-Reyes: \url{https://orcid.org/0000-0002-7004-1794}

\end{document}